\documentclass[a4paper,12pt,twoside]{amsart}
\usepackage{vmargin}
\usepackage{amssymb}
\usepackage{amsmath}
\usepackage{mathrsfs}
\usepackage[dvips]{graphicx}
\usepackage{vmargin}
\usepackage{color}
\usepackage{amscd}
\DeclareGraphicsExtensions{.pdf, .png, .jpg, .mps}
\usepackage[colorlinks=false]{hyperref}
\definecolor{myc}{cmyk}{0.0009,0.8,0.8,0.00}

 \numberwithin{equation}{section}

\newcommand{\N}{{\mathbb N}}

\newcommand{\Z}{{\mathbb Z}}
\newcommand{\R}{{\mathbb R}}

\def\norm2#1{\Vert#1\Vert_{L^2(\R^2;dx)}}
\def\p{\partial}

\newtheorem{Thm}{Theorem}[section]
\newtheorem{thm}[Thm]{Theorem}

\newtheorem{lem}[Thm]{Lemma}
\newtheorem{prop}[Thm]{Proposition}
\newtheorem{defi}[Thm]{Definition}

\begin{document}

\title{Structure constants of the Weyl calculus}
\author{Wen DENG}
\begin{address}{Wen Deng, Institut de Math\'ematiques de Jussieu, Universit\'e Pierre-et-Marie-Curie (Paris 6), 4 place Jussieu, 75005 Paris, France.}
\email{wendeng@math.jussieu.fr}
\urladdr{http://www.math.jussieu.fr/~wendeng/}
\end{address}

\subjclass[2000]{ }
\date{\today}
\begin{abstract} 
We find some explicit bounds on the ${\mathcal L}(L^2)$-norm of pseudo-differential operators with symbols defined by a metric on the phase space. In particular, we prove that this norm depends only on the ``structure constants" of the metric and a fixed semi-norm of the symbol. Analogous statements are made for the Fefferman-Phong inequality.
\end{abstract}
\maketitle

\section{Introduction}

The class of symbols $S^m_{1,0}$ consists of smooth functions $a$ defined on the phase space $\R^{n}\times\R^n$ such that for all multi-indices $\alpha,\beta$,
\begin{equation}\label{eq.1}
 |(\p_\xi^\alpha\p_x^\beta a)(x,\xi)|\leq C_{\alpha,\beta}(1+|\xi|)^{m-|\alpha|}.
\end{equation}
The best constants $C_{\alpha,\beta}$ in \eqref{eq.1} are called the semi-norms of the symbol $a$ in the Fr\'echet space $S^m_{1,0}$. 
We have

\medskip

{\it {\bf Property A.} If $a$ is in $S^0_{1,0}$, then $a(x,D)$ defines a bounded operator on $L^2(\R^n)$.}

\medskip

One might ask some very natural questions: the operator norm $\|a(x,D)\|_{{\mathcal L}(L^2(\R^n))}$ is bounded by which constant? Is it a semi-norm of the symbol $a$? If yes, then which semi-norm?
Questions of the same type might be asked for the constant $C$ in the following inequality:

\medskip

{\it {\bf Property B (Fefferman-Phong inequality).} If $a$ is a non-negative symbol belonging to $S^2_{1,0}$, then there exists $C>0$ such that, for all $u\in {\mathcal S}(\R^n)$,}
\begin{equation}\label{eq.2}
{\rm Re}\langle a(x,D)u,u\rangle_{L^2(\R^n)}+C\|u\|_{L^2(\R^n)}^2\geq0.
\end{equation}

\medskip

We can pose similar questions in many other examples of classes of symbols, such as the semi-classical symbols, Shubin's class, etc. As a particular example, the class $\Sigma^m$, defined as the set of smooth functions $a$ on $\R^{n}\times\R^n\times\R^+$ such that for all multi-indices $\alpha,\beta$,
\begin{equation}\label{eq.3}
\forall x ,\xi\in\R^n,\ \tau\in\R^+,\quad|(\p_\xi^\alpha\p_x^\beta a)(x,\xi,\tau)|\leq C_{\alpha,\beta}(1+|\xi|+\tau)^{m-|\alpha|},
\end{equation} 
is useful for Carleman estimates. One would like to check the Property A and Property B independent of the parameter $\tau$.

Several authors like Bony \cite{bony}, Boulkhemair \cite{boulkhemair}, Lerner-Morimoto \cite{LM}, have already considered these questions and they were able to identify the constants.
The constants in Properties A, B are always a constant $C_n$ times a semi-norm of the symbol, whose order depends only on the dimension $n$.
Although the problem is well-understood for a single class of pseudo-differential calculus, including the class $S(m,g)$ developed by H\"ormander, we want to address a more general and useful question, having in mind the class $\Sigma^m$ depending on the non-compact parameter $\tau\geq0$ which is defined in \eqref{eq.3} and is useful for Carleman estimates. 

In this paper, we consider the Weyl quantization for pseudodifferential operators and we choose the framework with a metric $g$ on the phase space. The metric $g$ is assumed to be admissible, that is slowly varying, satisfying the uncertainty principle and is temperate (see Definition \ref{def.s.m}, \ref{def.metric} below). The so-called structure constants of $g$ are closely related to these properties.  
We can define very general classes of symbols $S(m,g)$ attached to the metric $g$ and a $g$-admissible weight $m$ (see Definition \ref{def.symbol}) and we have an effective symbolic calculus.
The following results are classical: (see \cite[chapter 18]{hor}, \cite[chapter 2]{lerner})
\begin{eqnarray}\label{eq.4}
\text{\it $L^2$-boundedness:}\ &a\in S(1,g)&\Longrightarrow \|a^w\|_{{\mathcal L}(L^2(\R^n))}\leq C,\\
\label{eq.5}\text{\it Fefferman-Phong:}\ & a\in S(\lambda_g^2,g),\ a\geq0 &\Longrightarrow a^w+C\geq0.
\end{eqnarray}
The question that we would like to address is the following: what happens if we change the metric $g$ but keep the same structure constants?

We intend to show that the constants involved in \eqref{eq.4}, \eqref{eq.5} depend only on the structure constants of the metric $g$ and a fixed semi-norm of $a$. 
 Since it may happen that the metric $g$ depends on a non-compact parameter with uniform structure constants (e.g. the class $\Sigma^m$), this fact is useful explicitly or implicitly in many examples where these metrics are used and it seems useful to rely on a more stable argument than referring to ``inspection of the proofs".

\smallskip

{\bf Remark.} An abstract functional analysis argument does not seem to work.
Our method is to follow the proofs, by carefully computing all the constants.

\section{Metric on the phase space}
In this section, we introduce the definitions of the admissible metric and exhibit its properties.
We use the Weyl quantization which associates to a symbol $a$ the operator $a^w$ defined by
\begin{equation}\label{eq.weyl}
(a^wu)(x)=\iint e^{2i\pi(x-y)\cdot\xi}a(\frac{x+y}{2},\xi)u(y)dyd\xi.
\end{equation}
Consider the symplectic space $\R^{2n}$ equipped with the symplectic form $\sigma=\sum_{j=1}^nd\xi^j\wedge dx^j.$
Given a positive-definite quadratic form $\Gamma$ on $\R^{2n}$, we define
\begin{equation}\label{eq.dual.met}
\Gamma^\sigma(T)=\sup_{\Gamma(Y)=1}\sigma(T,Y)^2,
\end{equation}
which is also a positive-definite quadratic form. 
Let $g$ be a measurable map from $\R^{2n}$ into the cone of positive-definite quadratic forms on $\R^{2n}$, i.e. for each $X\in\R^{2n}$, $g_X$ is a positive definite quadratic form on $\R^{2n}$.
\begin{defi}[Slowly varying metric]\label{def.s.m}
We say that $g$ is a slowly varying metric on $\R^{2n}$, if there exists $C_0\geq1$ such that for all $X,Y,T\in\R^{2n}$,
\begin{equation}\label{eq.slow}
g_X(X-Y)\leq C_0^{-1}\Longrightarrow C_0^{-1}\leq \frac{g_X(T)}{g_Y(T)}\leq C_0.
\end{equation}
\end{defi}
\begin{defi}[Slowly varying weight]\label{def.s.w}
Let $g$ be a slowly varying metric on $\R^{2n}$. A function $m\colon\R^{2n}\to(0,+\infty)$ is called a $g$-slowly varying weight if there exists $\mu_m\geq1$ such that for all $X,Y\in\R^{2n}$,
\begin{equation}\label{eq.w.slow}
g_X(Y-X)\leq \mu_m^{-1}\Longrightarrow \mu_m^{-1}\leq\frac{m(X)}{m(Y)}\leq \mu_m.
\end{equation}
\end{defi}
\begin{defi}[Class of symbols]\label{def.symbol}
Let $g$ be a slowly varying metric on $\R^{2n}$ and $m$ be a $g$-slowly varying weight. The class of  symbols $S(m,g)$ is defined as the subset of functions $a\in C^\infty(\R^{2n})$ satisfying that for all $k\in\N$, there exists $C_k>0$ such that for all $X,T_1,\cdots,T_k\in\R^{2n}$, 
$$|a^{(k)}(X)(T_1,\cdots,T_k)|\leq C_km(X)\prod_{1\leq j\leq k}g_X(T_j)^{1/2}.$$
For $a\in S(m,g)$, $l\in\N$, we denote
\begin{equation}\label{seminorm}
\|a\|_{S(m,g)}^{(l)}=\max_{0\leq k\leq l}\sup_{\substack{X,T_j\in\R^{2n}\\ g_X(T_j)=1}}|a^{(k)}(X)(T_1,\cdots,T_k)|m(X)^{-1}.
\end{equation}
The space $S(m,g)$ equipped with the countable family of semi-norms $(\|\cdot\|_{S(m,g)}^{(l)})_{l\in\N}$ is a Fr\'echet space.
\end{defi}
For a slowly varying metric $g$ on the phase space $\R^{2n}$, we can introduce some partition of unity related to $g$. Define the $g$-ball near $X\in\R^{2n}$
\begin{equation}\label{def.ball}
U_{X,r}=\{Y,g_X(X-Y)\leq r^2\},
\end{equation}
we have the following theorem, which is Theorem 2.2.7 in \cite{lerner}.
\begin{thm}[Partition of unity]\label{thm.partition}
Let $g$ be a slowly varying metric on $\R^{2n}$ and $C_0>0$ given in (\ref{eq.slow}). Then for all $r\in(0,C_0^{-1/2}]$, there exists a family $(\varphi_Y)_{Y\in\R^{2n}}$ of smooth functions supported in $U_{Y,r}$ such that 
\begin{equation}\label{eq.sn.parti}
\forall k\in\N, \quad \sup_{Y\in\R^{2n}}\|\varphi_Y\|_{S(1,g)}^{(k)}\leq C(k,r,n,C_0),
\end{equation}
\begin{equation}\label{eq.parti}
\forall X\in\R^{2n},\quad\int_{\R^{2n}}\varphi_Y(X)|g_Y|^{1/2}dY=1,
\end{equation}
where $C(k,r,n,C_0)$ is a positive constant depending only on $k,r,n,C_0$ and  $|g_Y|$ is the determinant of $g_Y$ with respect to the standard Euclidean norm.
\end{thm}
\begin{proof}
As in the proof of Theorem 2.2.7 in \cite{lerner}, let $\chi_0\in C_0^\infty(\R_+;[0,1])$ non-increasing such that $\chi_0(t)=1$ on $t\leq1/2$, $\chi_0(t)=0$ on $t\geq1$. Define for $r\in(0,C_0^{-1/2}]$,
$$\omega(X,r)=\int_{\R^{2n}}\underbrace{\chi_0\big(r^{-2}g_Y(X-Y)\big)}_{=\omega_Y(X)}|g_Y|^{1/2}dY.$$
Since $\omega_Y(X)$ is supported in $U_{Y,r}$ and $\chi_0$ is non-increasing, by \eqref{eq.slow} we have
$$\omega(X,r)\geq\int_{R^{2n}}\chi_0\big(r^{-2}C_0g_X(X-Y)\big)C_0^{-n}|g_X|^{1/2}dY=\int_{\R^{2n}}\chi_0(|Z|^2)dZC_0^{-2n}r^{2n},$$
and an estimate from above of the same type, i.e. there exists a positive constant $C_1=C_1(r,n,C_0)$ such that
$$C_1^{-1}\leq \omega(X,r)\leq C_1.$$
Now let us check the derivatives of $\omega_Y(X)$. Using the notation $\langle,\rangle_Y$ the inner-product associated to $g_Y$, we have
$$\omega_Y'(X)T=\chi_0'\big(r^{-2}g_Y(X-Y)\big)r^{-2}\langle X-Y,T\rangle_Y,$$
and by induction, for $k\geq1$, $T\in\R^{2n}$, $\omega_Y^{(k)}(X)T^k$ is a finite sum of terms of type
\begin{equation}\label{eq.parti01}
c_{p,k}\chi_0^{(p)}\big(r^{-2}g_Y(X-Y)\big)r^{-2p}\langle X-Y,T\rangle_Y^{2p-k}g_Y(T)^{k-p},
\end{equation}
where $c_{p,k}$ is a constant depending only on $p,k$ and $p\in[k/2,k]\cap\N$. Since the support of $\chi_0^{(p)}$ is included in $[0,1]$ and $r^2\leq C_0^{-1}$, the term \eqref{eq.parti01} can be bounded from above by
$$c_{p,k}\|\chi_0^{(p)}\|_{L^\infty}r^{-2p}(r^2)^{(2p-k)/2}C_0^{k/2}g_X(T)^{k/2},$$
so that for all $k\geq1$, $|\omega_Y^{(k)}(X)T^k|\leq C(k,r,C_0)g_X(T)^{k/2}$. This implies that $\omega_Y$ is in $S(1,g)$ and moreover,  
\begin{equation}\label{eq.parti02}
\forall k\in\N,\quad\sup_{Y\in\R^{2n}}\|\omega_Y\|_{S(1,g)}^{(k)}\leq C(k,r,C_0).
\end{equation}
Now we choose a non-negative function $\chi_1\in C_0^\infty(\R_+;[0,1])$ such that $\chi_1(t)=1$ on $t\leq1$, then
\begin{align*}
|\omega^{(k)}(X,r)T^k|&=\big|\int_{\R^{2n}}\omega_Y^{(k)}(X)T^{k}\chi_1\big(r^{-2}g_Y(X-Y)\big)|g_Y|^{1/2}dY\big|\\
&\leq\sup_{Y\in\R^{2n}}\|\omega_Y\|_{S(1,g)}^{(k)}g_X(T)^{k/2}\int_{\R^{2n}}\chi_1\big(r^{-2}g_Y(X-Y)\big)|g_Y|^{1/2}dY\\
&\leq C(k,r,n,C_0)g_X(T)^{k/2},
\end{align*} 
which implies that $\omega(\cdot,r)$ is a symbol in $S(1,g)$ with $\|\omega(\cdot,r)\|_{S(1,g)}^{(k)}\leq C'(k,r,n,C_0)$.
Since $\omega$ is bounded from below by $C_1^{-1}$, the function $\omega(\cdot,r)^{-1}$ is also in $S(1,g)$ and 
\begin{equation}\label{eq.parti03}
\|\omega(\cdot,r)^{-1}\|_{S(1,g)}^{(k)}\leq C''(k,r,n,C_0).
\end{equation}
We define 
$$\varphi_Y(X)=\omega_Y(X)\omega(X,r)^{-1},$$
then the estimate \eqref{eq.sn.parti} follows from \eqref{eq.parti02}, \eqref{eq.parti03} and moreover, the family $(\varphi_Y)_{Y\in\R^{2n}}$ satisfies the requirements of Theorem \ref{thm.partition}. 
\end{proof}
A direct consequence of Theorem \ref{thm.partition} is the following.
\begin{prop}\label{prop.parti.a}
Let $g$ be a slowly varying metric on $\R^{2n}$ and $m$ be a $g$-slowly varying weight. Let $C_0,\mu_m$ be given in (\ref{eq.slow}), (\ref{eq.w.slow}) respectively. Let $a$ be a symbol in $S(m,g)$. Then for all $0<r\leq \min(C_0^{-1/2},\mu_m^{-1/2})$, 
$$a(X)=\int_{\R^{2n}}a_Y(X)|g_Y|^{1/2}dY,$$
where $a_Y$ has support included in $U_{Y,r}$ and
\begin{equation}\label{eq.parti.a}
\forall k\in\N,\quad\sup_{Y\in\R^{2n}}\|a_Y\|_{S(m(Y),g_Y)}^{(k)}\leq C(k,r,C_0,n,\mu_m)\|a\|_{S(m,g)}^{(k)}.
\end{equation}
\end{prop}
\begin{proof}
Define $a_Y(X)=a(X)\varphi_Y(X)$. Since $\varphi_Y$ is supported in $U_{Y,r}$, we have, for $k\geq0$, $X\in U_{Y,r}$, $T\in\R^{2n}$,
\begin{align*}
|a_Y^{(k)}(X)T^k|&=\Big|\sum_{0\leq l\leq k}\binom{k}{l}a^{(l)}(X)T^l\cdot\varphi_Y^{(k-l)}(X)T^{k-l}\Big|\\
&\leq \sum_{0\leq l\leq k}c_{k,l}\|a\|_{S(m,g)}^{(l)}m(X)g_X(T)^{l/2}\|\varphi_Y\|_{S(1,g)}^{(k-l)}g_X(T)^{(k-l)/2}\\
&\leq C(k)\|a\|_{S(m,g)}^{(k)}\|\varphi_Y\|_{S(1,g)}^{(k)}m(X)g_X(T)^{k/2}\\
&\leq C(k)\mu_mC_0^{k/2}\|a\|_{S(m,g)}^{(k)}\|\varphi_Y\|_{S(1,g)}^{(k)}m(Y)g_Y(T)^{k/2},
\end{align*}
which completes the proof.
\end{proof}

For two positive-definite quadratic forms $\Gamma_1,\Gamma_2$ on $\R^{2n}$, the harmonic mean $\Gamma_1\wedge \Gamma_2$ is defined by
\begin{equation}\label{eq.harm}
\Gamma_1\wedge \Gamma_2=2(\Gamma_1^{-1}+\Gamma_2^{-1})^{-1},
\end{equation}
which is also a positive-definite quadratic form on $\R^{2n}$.

\begin{defi}[Admissible metric]\label{def.metric}
We say that $g$ is an admissible metric on $\R^{2n}$ if $g$ is slowly varying (see Definition \ref{def.s.m}) and there exist $C_0'>0$, $N_0\in\N$ such that for all $X,Y,T\in\R^{2n}$, 
\begin{align}
\label{eq.uncertain}&\text{uncertainty principle }& & g_X(T)\leq g_X^\sigma(T),\\
\label{eq.temper}&\text{temperance }& & g_X(T)\leq C_0'g_Y(T)\big(1+(g_X^\sigma\wedge g_Y^\sigma)(X-Y)\big)^{N_0},
\end{align}
where $g^\sigma$ is given by (\ref{eq.dual.met}) and $\wedge$ given by (\ref{eq.harm}).

We may suppose $C_0'=C_0$ in the sequel, where $C_0$ is given in (\ref{eq.slow}). Then the constants $(C_0,N_0)$ appearing in (\ref{eq.slow}), (\ref{eq.temper}) are called the {\rm\bf structure constants} of the metric $g$.
\end{defi}

\begin{defi}[Admissible weight]\label{def.weight}
Suppose that $g$ is an admissible metric on $\R^{2n}$. A function $m\colon\R^{2n}\to(0,+\infty)$ is called a $g$-admissible weight if $m$ is a $g$-slowly varying weight (see Definition \ref{def.s.w}) and there exist $\mu_m>0$, $\nu_m\in\N$ such that for all $X,Y\in\R^{2n}$,
\begin{align}
\label{eq.w.temper}&m(X)\leq \mu_mm(Y)\big(1+(g_X^\sigma\wedge g_Y^\sigma)(X-Y)\big)^{\nu_m}.
\end{align}
The constants $(\mu_m,\nu_m)$ appearing in (\ref{eq.w.slow}), (\ref{eq.w.temper}) are called the structure constants of the $g$-admissible weight $m$.
\end{defi}

Let $g$ be an admissible metric on $\R^{2n}$. We define for $X\in\R^{2n}$,
\begin{equation}\label{def.lambda}
\lambda_g(X)=\inf_{T\neq0}\Big(\frac{g_X^\sigma(T)}{g_X(T)}\Big)^{1/2}.
\end{equation}
Then the uncertainty principle \eqref{eq.uncertain} can be expressed by 
$$g_X\leq\lambda_g(X)^{-2}g_X^\sigma,\quad \lambda_g(X)\geq1.$$

\begin{lem}[{\cite[Remark 2.2.17]{lerner}}]\label{lem.lambda}
For any $s\in\R$, $\lambda_g^s$ is an admissible weight, with structure constants $(\mu_{\lambda_g^s},\nu_{\lambda_g^s})$ in (\ref{eq.w.slow}), (\ref{eq.w.temper}) depending only on the structure constants of the metric $g$ $(C_0,N_0)$.
\end{lem}
\begin{proof}
We first verify that $\lambda_g^s$ is a $g$-slowly varying weight. For $g_X(X-Y)\leq C_0^{-1}$, $T\in\R^{2n}$, we have
$$C_0^{-1}g_X(T)\leq g_Y(T)\leq C_0g_X(T),\quad C_0^{-1}g_X^\sigma(T)\leq g_Y^\sigma(T)\leq C_0g_X^\sigma(T),$$
which implies
$$C_0^{-2}\frac{g_X^\sigma(T)}{g_X(T)}\leq\frac{g_Y^\sigma(T)}{g_Y(T)}\leq C_0^2\frac{g_X^\sigma(T)}{g_X(T)}.$$
Taking the infimum with respect to $T$, we get 
$$C_0^{-2}\lambda_g(X)^2\leq\lambda_g(Y)^2\leq C_0^2\lambda_g(X)^2,$$ 
so that $\lambda_g$ is $g$-slowly varying with $\mu_{\lambda_g}=C_0$ and so is $\lambda_g^s$ with $\mu_{\lambda_g^s}=C_0^{|s|}$. Next we check that $\lambda_g^s$ is temperate. We have for all $X,Y,T\in\R^{2n}$,
\begin{align*}
g_X(T)&\geq C_0^{-1}g_Y(T)\big(1+(g_X^\sigma\wedge g_Y^\sigma)(X-Y)\big)^{-N_0},\\
g_X^\sigma(T)&\leq C_0g_Y^\sigma(T)\big(1+(g_X^\sigma\wedge g_Y^\sigma)(X-Y)\big)^{N_0},
\end{align*}
which gives 
$$\lambda_g(X)^2\leq C_0^2\lambda_g(Y)^2\big(1+(g_X^\sigma\wedge g_Y^\sigma)(X-Y)\big)^{2N_0}.$$
Thus $\lambda_g$ is temperate with $\nu_{\lambda_g}=N_0$ and so is $\lambda_g^s$ with $\nu_{\lambda_g^s}=|s|N_0$. This completes the proof of Lemma \ref{lem.lambda}.
\end{proof}

The composition $a\sharp b$ of two symbols is defined by $a^wb^w=(a\sharp b)^w$ and we have, with the notations $[X,Y]=\sigma(X,Y)$, $D=(2i\pi)^{-1}\partial$, 
\begin{equation}\label{eq.compo1}
(a\sharp b)(X)=2^{2n}\iint_{\R^{2n}\times\R^{2n}}a(Y)b(Z)e^{-4i\pi[X-Y,X-Z]}dYdZ,
\end{equation} 
\begin{equation}\label{eq.compo2}
(a\sharp b)(X)=\exp\big(i\pi[D_Y,D_Z]\big)\big(a(Y)b(Z)\big)_{|Y=Z=X}.
\end{equation}
For $a\in S(m_1,g)$, $b\in S(m_2,g)$, we have the asymptotic expansion
\begin{equation}
\label{eq.compo.asy} a\sharp b(x,\xi)=\sum_{0\leq k<p}w_k(a,b)(x,\xi)+r_p(a,b)(x,\xi),
\end{equation}
\begin{align}
\label{eq.compo.asy1}
\text{with}\quad w_k(a,b)=2^{-k}\sum_{|\alpha|+|\beta|=k}\frac{(-1)^{|\beta|}}{\alpha!\beta!}D_\xi^\alpha\partial_x^\beta a\ D_\xi^\beta\partial_x^\alpha b \ &\in S(m_1m_2\lambda_g^{-k},g),\\
\label{eq.compo.asy2}
r_p(a,b)(X)=\quad R_p\big(a(X)\otimes b(Y)\big)_{|Y=X} \ &\in S(m_1m_2\lambda_g^{-p},g), 
\end{align}
\begin{equation}
\label{eq.compo.asy3}
R_p=\int_0^1\frac{(1-\theta)^{p-1}}{(p-1)!}\exp\frac{\theta}{4i\pi}[\partial_X,\partial_Y]d\theta\Big(\frac{1}{4i\pi}[\partial_X,\partial_Y]\Big)^p
\end{equation}
Notice $w_1(a,b)=\frac{1}{4i\pi}\{a,b\}$, where $\{,\}$ denotes the Poisson bracket, so that the asymptotic \eqref{eq.compo.asy} at $p=2$ is 
\begin{equation}\label{eq.compo.asy4}
a\sharp b=ab+\frac{1}{4i\pi}\{a,b\}+r_2(a,b).
\end{equation}

\begin{defi}[The main distance function]
Let $g$ be an admissible metric on $\R^{2n}$. Define the main distance function, for $r>0$, $X,Y\in\R^{2n}$,
\begin{equation}\label{eq.dis}
\delta_r(X,Y)=1+(g_X^\sigma\wedge g_Y^\sigma)(U_{X,r}-U_{Y,r}),
\end{equation}
where $U_{X,r}$ is given in \eqref{def.ball} and
$$g(U-V)=\inf_{X\in U,Y\in V}g(X-Y).$$
\end{defi}

\begin{lem}[{\cite[Lemma 2.2.24]{lerner}}, Integrability of $\delta_r$]\label{lem.distance}
Let $g$ be an admissible metric with structure constants $(C_0,N_0)$. Then there exist positive constants $N_1=N_1(n,C_0,N_0)$, $C=C(n,C_0,N_0)$ such that for all $r\in(0,C_0^{-1/2}]$,
\begin{equation}\label{eq.dis01}
\sup_{X\in\R^{2n}}\int_{\R^{2n}}\delta_r(X,Y)^{-N_1}|g_Y|^{1/2}dY\leq C<+\infty,
\end{equation}

\end{lem}
\begin{proof} Suppose $r\leq C_0^{-1/2}$.
Using the slowness and temperance of $g$, for $X'\in U_{X,r}$, $Y'\in U_{Y,r}$, $T\in\R^{2n}$, we have
\begin{align*}
(g_X^\sigma\wedge g_Y^\sigma)(T)\geq C_0^{-1}(g_{X'}^\sigma\wedge g_{Y'}^\sigma)(T)&\geq C_0^{-2}g_{X'}^\sigma(T)\big(1+(g_{X'}^\sigma\wedge g_{Y'}^\sigma)(X'-Y')\big)^{-N_0}\\
&\geq C_0^{-3}g_{X}^\sigma(T)\big(1+C_0(g_{X}^\sigma\wedge g_{Y}^\sigma)(X'-Y')\big)^{-N_0}\\
&\geq C_0^{-3-N_0}g_{X}^\sigma(T)\big(1+(g_{X}^\sigma\wedge g_{Y}^\sigma)(X'-Y')\big)^{-N_0}.
\end{align*}
Taking the infimum in $X'\in U_{X,r}$, $Y'\in U_{Y,r}$, we get 
\begin{equation}\label{eq.dis03}
g_X^\sigma(T)\leq C_0^{3+N_0}\delta_r(X,Y)^{N_0}(g_X^\sigma\wedge g_Y^\sigma)(T).
\end{equation}
We have also
\begin{align*}
\frac{g_X(T)}{g_Y(T)}\leq C_0^2\frac{g_{X'}(T)}{g_{Y'}(T)}&\leq C_0^3\big(1+(g_{X'}^\sigma\wedge g_{Y'}^\sigma)(X'-Y')\big)^{N_0}\\
&\leq C_0^3\big(1+C_0(g_{X}^\sigma\wedge g_{Y}^\sigma)(X'-Y')\big)^{N_0}\\
&\leq C_0^{3+N_0}\big(1+(g_{X}^\sigma\wedge g_{Y}^\sigma)(X'-Y')\big)^{N_0}.
\end{align*}
By taking the infimum in $X', Y'$, we get the following inequality 
\begin{equation}\label{eq.dis02}
\frac{g_X(T)}{g_Y(T)}\leq C_0^{3+N_0}\delta_r(X,Y)^{N_0}.
\end{equation}
Then
\begin{align*}
1+&g_X(X-Y)\leq 1+3g_X(X-X')+3g_X(X'-Y')+3g_X(Y'-Y)\\
&\leq 3C_0^{3+N_0}\delta_r(X,Y)^{N_0}\big(1+g_X(X-X')+g_X(X'-Y')+g_Y(Y'-Y)\big)\quad\text{\small by \eqref{eq.dis02}}\\
&\leq 3C_0^{3+N_0}\delta_r(X,Y)^{N_0}\big(1+2r^2+g_X^\sigma(X'-Y')\big)\\
&\leq 9C_0^{6+2N_0}\delta_r(X,Y)^{2N_0}\big(1+(g_X^\sigma\wedge g_Y^\sigma)(X'-Y')\big)\quad\text{\small by \eqref{eq.dis03}},
\end{align*}
so that $1+g_X(X-Y)\leq 9C_0^{6+2N_0}\delta_r(X,Y)^{2N_0+1}$. In the other hand, we have 
$$\frac{|g_Y|^{1/2}}{|g_X|^{1/2}}\leq C_0^{n(3+N_0)}\delta_r(X,Y)^{nN_0},$$
so that for $N_1=nN_0+(n+1)(2N_0+1)>0$,
\begin{align*}
\int_{\R^{2n}}\delta_r(X,Y)^{-N_1}&|g_Y|^{1/2}dY\leq C(n,C_0,N_0)\int_{\R^{2n}}\delta_r(X,Y)^{-N_1+nN_0}|g_X|^{1/2}dY\\
&\leq C'(n,C_0,N_0)\int_{\R^{2n}}\big(1+g_X(X-Y)\big)^{-(n+1)}|g_X|^{1/2}dY\\
&= C'(n,C_0,N_0)\int_{\R^{2n}}(1+|Z|^2)^{-(n+1)}dZ<+\infty.
\end{align*}
The proof of the lemma is complete.
\end{proof}


\section{$L^2$-boundedness}
In this section, we prove the $L^2$-boundedness of pseudo-differential operators with symbol in $S(1,g)$ and make precise the operator norms.
\subsection{The constant metric case}

\begin{prop}\label{prop.const}
Suppose that $g$ is a positive-definite quadratic form (constant metric) on $\R^{2n}$ with $g\leq g^\sigma$. Then there exists a constant $C(n)>0$ depending only on the dimension $n$ such that for all $a\in S(1,g)$, 
$$\|a^w\|_{{\mathcal L}(L^2(\R^n))}\leq C(n)\|a\|_{S(1,g)}^{(2n+1)}.$$ 
\end{prop}
\begin{proof}
Since $g$ is a constant metric, according to Lemma 4.4.25 in \cite{lerner}, there exist symplectic coordinates $(x,\xi)$ such that 
$$g=\sum_{1\leq j\leq n}\lambda_j^{-1}(|dx_j|^2+|d\xi_j|^2),\quad g^\sigma=\sum_{1\leq j\leq n}\lambda_j(|dx_j|^2+|d\xi_j|^2),$$
with $\lambda_j>0$. $g\leq g^\sigma$ is expressed as
$$\min_{1\leq j\leq n}\lambda_j\geq1.$$
As a result, we have $g\leq|dx|^2+|d\xi|^2:=\Gamma_0$, which implies $S(1,g)\subset S(1,\Gamma_0)$ and for all $a\in S(1,g)$, 
\begin{equation}\label{eq.cm1}
\forall l\in\N,\quad\|a\|_{S(1,\Gamma_0)}^{(l)}\leq\|a\|_{S(1,g)}^{(l)}.
\end{equation}
By Theorem 1.1.4 in \cite{lerner} and $a^w=(J^{1/2}a)(x,D)$, where $J^t$ is introduced in Lemma 4.1.2 in \cite{lerner}, we obtain that
$$\|a^w\|_{{\mathcal L}(L^2(\R^n))}\leq C(n)\|a\|_{S(1,\Gamma_0)}^{(2n+1)},$$
where $C(n)$ depends only on $n$. Together with \eqref{eq.cm1}, we complete the proof of the proposition.
\end{proof}

\subsection{The general case}

\begin{thm}\label{thm.continuity}
Let $g$ be an admissible metric on $\R^{2n}$ with structure constants $(C_0,N_0)$ (see Definition \ref{def.metric}). Then there exist $C=C(n,C_0,N_0)>0$ and $l=l(n,C_0,N_0)\in\N$ such that for all $a\in S(1,g)$ (see Definition \ref{def.symbol}), 
$$\|a^w\|_{{\mathcal L}(L^2(\R^n))}\leq C\|a\|_{S(1,g)}^{(l)}.$$
\end{thm}
\begin{proof}
Using the partition in Proposition \ref{prop.parti.a}, we write
$$a^w=\int_{\R^{2n}}a_Y^w|g_Y|^{1/2}dY,$$
where $a_Y$ is supported in $U_{Y,r}$ and satisfies \eqref{eq.parti.a}.
By Proposition \ref{prop.const}, we have $\sup_Y\|a_Y^w\|_{{\mathcal L}(L^2(\R^n))}\leq C(r,n,C_0,N_0)\|a\|_{S(1,g)}^{(2n+1)}<+\infty$. The following lemma is useful.

\begin{lem}[Cotlar]\label{lem.cotlar}
Let $H$ be a Hilbert space and $(\Omega,{\mathcal A},\nu)$ a measured space such that $\nu$ is a $\sigma$-finite positive measure. Let $(A_y)_{y\in\Omega}$ be a measurable family of bounded operators on $H$ such that 
$$\sup_{y\in\Omega}\int_\Omega \|A_y^\ast A_z\|_{{\mathcal L}(H)}^{1/2}d\nu(z)\leq M,\quad\sup_{y\in\Omega}\int_\Omega \|A_y A_z^\ast\|_{{\mathcal L}(H)}^{1/2}d\nu(z)\leq M.$$
Then for all $u\in H$, we have
$$\iint_{\Omega\times\Omega}|\langle A_yu,A_zu\rangle_H|d\nu(y)d\nu(z)\leq M^2\|u\|_H^2,$$
which implies the strong convergence of $A=\int_\Omega A_yd\nu(y)$ and $\|A\|_{{\mathcal L}(H)}\leq M$.
\end{lem}

In order to apply Cotlar's lemma, we should estimate $\|\bar a_Y^wa_Z^w\|_{{\mathcal L}(L^2(\R^n))}$, i.e. a semi-norm of $\bar a_Y\sharp a_Z$ in $S(1,g_Y+g_Z)$. Indeed, the following estimate holds.

\begin{lem}\label{lem.1}
Let $g$, $a_Y$ be as above. For any $k, N\in\N$, there exist $C=C(k,N,n)>0$, $l=l(k,N,n)\in\N$ such that
\begin{equation}\label{eqlem.1}
\|\bar a_Y\sharp a_Z\|_{S(1,g_Y+g_Z)}^{(k)}\leq C \|\bar a_Y\|_{S(1,g_Y)}^{(l)}\|a_Z\|_{S(1,g_Z)}^{(l)}\delta_r(Y,Z)^{-N}.
\end{equation}
\end{lem}

We use some biconfinement estimates, which can be found in \cite[section 2.3]{lerner}, to prove Lemma \ref{lem.1}.

\begin{defi}[Confined symbols]\label{def.conf}
Let $g$ be a positive-definite quadratic form on $\R^{2n}$ such that $g\leq g^\sigma$. Let $a$ be a smooth function on $\R^{2n}$ and $U\subset \R^{2n}$. We say that $a$ is $g$-confined in $U$, if for all $k,N\in\N$, there exits $C_{k,N}>0$ such that for all $X,T\in\R^{2n}$,
$$|a^{(k)}(X)T^k|\leq C_{k,N}g(T)^{k/2}\big(1+g^\sigma(X-U)\big)^{-N/2}.$$
We denote
\begin{equation}\label{def.conf.1}
\|a\|_{g,U}^{(k,N)}=\sup_{X,T\in\R^{2n},g(T)=1}|a^{(k)}(X)T^k|\big(1+g^\sigma(X-U)\big)^{N/2},
\end{equation}
\begin{equation}\label{def.conf.2}
\text{and}\qquad |\!|\!| a|\!|\!|_{g,U}^{(l)}=\max_{k\leq l}\|a\|_{g,U}^{(k,l)}.
\end{equation}
\end{defi}

\begin{thm}[{\cite[Theorem 2.3.2]{lerner}}, biconfinement estimate]\label{thm.biconf}
Let $g_1,g_2$ be two positive-definite quadratic forms on $\R^{2n}$ such that $g_j\leq g_j^\sigma$. Let $a_j,j=1,2$ be $g_j$-confined in $U_j$, a $g_j$-ball of radius $\leq1$. Then for all $k,N\in\N$, for all $X,T\in\R^{2n}$,
\begin{equation}\label{eq.biconf}
|(a_1\sharp a_2)^{(k)}(X)T^k|\leq A_{k,N}(g_1+g_2)(T)^{k/2}\Big(1+(g_1^\sigma\wedge g_2^\sigma)(X-U_1)+(g_1^\sigma\wedge g_2^\sigma)(X-U_2)\Big)^{-N/2},
\end{equation}
with $A_{k,N}=\gamma(k,N,n)|\!|\!| a_1|\!|\!|_{g_1,U_1}^{(l)}|\!|\!|a_2|\!|\!|_{g_2,U_2}^{(l)}$, $l=2n+1+k+N$.
\end{thm}

Now we begin the proof of Lemma \ref{lem.1}.
\begin{proof}[Proof of Lemma \ref{lem.1}]
The symbol $a_Y$ is $g_Y$-confined in $U_{Y,r}$, since $a_Y$ is supported in the $g_Y$-ball $U_{Y,r}$. Moreover, we have
$$\forall k,N\in\N,\qquad \|a_Y\|_{g_Y,U_{Y,r}}^{(k,N)}=\sup_{\substack{X\in U_{Y,r},T\in\R^{2n}\\ g_Y(T)=1}}|a^{(k)}(X)T^k|,$$
$$\forall l\in\N,\qquad|\!|\!| a_Y|\!|\!|_{g_Y,U_{Y,r}}^{(l)}=\max_{k\leq l}\|a_Y\|_{g_Y,U_{Y,r}}^{(k,l)}=\|a_Y\|_{S(1,g_Y)}^{(l)}.$$
Applying \eqref{eq.biconf} to $\bar a_Y\sharp a_Z$ and using the triangular inequality
$$(g_Y^\sigma\wedge g_Z^\sigma)(X-U_{Y,r})+(g_Y^\sigma\wedge g_Z^\sigma)(X-U_{Z,r})\geq\frac{1}{2}(g_Y^\sigma\wedge g_Z^\sigma)(U_{Y,r}-U_{Z,r}),$$
we get 
\begin{align*}
|(\bar a_Y\sharp a_Z)^{(k)}(X)T^k|&\leq \gamma(k,N,n)\|\bar a_Y\|_{S(1,g_Y)}^{(l)}\|a_Z\|_{S(1,g_Z)}^{(l)}(g_Y+g_Z)(T)^{k/2}\\
&\quad\times\big(1+\frac{1}{2}(g_Y^\sigma\wedge g_Z^\sigma)(U_{Y,r}-U_{Z,r})\big)^{-N/2}.
\end{align*}
Using the definition of the distance $\delta_r$, we complete the proof of Lemma \ref{lem.1}.
\end{proof}

\noindent{\it End of the proof of Theorem \ref{thm.continuity}.}
Now by Proposition \ref{prop.const}, Lemma \ref{lem.1} and the estimate \eqref{eq.parti.a}, we obtain that for any $N>0$, there exists $l=l(N,n)\in\N$ such that
\begin{align*}
\|\bar a_Y^wa_Z^w\|_{{\mathcal L}(L^2(\R^n))}&\leq C(n)\|\bar a_Y^wa_Z^w\|_{S(1,g_Y+g_Z)}^{(2n+1)}\\
&\leq C(N,n) \|\bar a_Y\|_{S(1,g_Y)}^{(l)}\|a_Z\|_{S(1,g_Z)}^{(l)}\delta_r(Y,Z)^{-N}\\
&\leq C(N,n,C_0)\big(\|a\|_{S(1,g)}^{(l)}\big)^2\delta_r(Y,Z)^{-N}
\end{align*}
The same inequality holds for $a_Y\sharp\bar a_Z$. Choose $N=2N_1$, where $N_1$ is given in \eqref{eq.dis01}, so that 
$$\max\{\|\bar a_Y^wa_Z^w\|_{{\mathcal L}(L^2(\R^n))}^{1/2},\|a_Y^w\bar a_Z^w\|_{{\mathcal L}(L^2(\R^n))}^{1/2}\}\leq C\|a\|_{S(1,g)}^{(l)}\delta_r(Y,Z)^{-N_1},$$
where $C=C(n,C_0,N_1)>0$, $l=l(n,N_1)\in\N$.
Then together with Lemma \ref{lem.distance}, the assumptions of Cotlar's lemma are fulfilled with $M=C\|a\|_{S(1,g)}^{(l)}$, and this completes the proof of Theorem \ref{thm.continuity}.
\end{proof}

\section{Fefferman-Phong inequality}
In this section, we prove that the constant in the Fefferman-Phong inequality depends only on the structure constants of the metric and a fixed semi-norm of the symbol.
\begin{thm}[Fefferman-Phong inequality]\label{thm.fp}
Let $g$ be an admissible metric on $\R^{2n}$ with structure constants $(C_0,N_0)$ (see Definition \ref{def.metric}). Let $a$ be  a non-negative symbol in $S(\lambda_g^2,g)$ (see Definition \ref{def.symbol} and (\ref{def.lambda})). Then the operator $a^w$ on $L^2(\R^n)$ is semi-bounded from below. More precisely, there exist $l=l(n,C_0,N_0)\in\N$, $C=C(n,C_0,N_0)>0$ such that
\begin{equation}
a^w+C\|a\|_{S(\lambda_g^2,g)}^{(l)}\geq0.
\end{equation}
\end{thm}

\subsection{The constant metric case}
For the constant metric case, we use the results of Sj\"ostrand and refer the readers to \cite[page 116]{lerner} for the detailed proof.

Let $1=\sum_{j\in \Z^{2n}}\chi_0(X-j)$ be a partition of unity, $\chi_0\in C_c^\infty(\R^{2n})$. Denote $\chi_j(X)=\chi_0(X-j)$.
\begin{prop}[{\cite[Proposition 2.5.6]{lerner}}]
Suppose $a\in {\mathcal S}(\R^{2n})$. We say that $a$ belongs to the class ${\mathcal A}$ if $\omega_a\in L^1(\R^{2n})$, with $\omega_a(\Xi)=\sup_{j\in\Z^{2n}}|{\mathcal F}(\chi_ja)(\Xi)|$, where ${\mathcal F}$ is the Fourier transform. We have
$$S_{0,0}^0\subset S_{0,0;2n+1}\subset {\mathcal A}\subset C^0(\R^{2n})\cap L^\infty(\R^{2n}),$$
where $S^0_{0,0}=C_b^\infty(\R^{2n})$ is the space of $C^\infty$ functions on $\R^{2n}$ which are bounded as well as all their derivatives, $S_{0,0;2n+1}$ is the set of functions defined on $\R^{2n}$ such that $|(\partial_\xi^\alpha\partial_x^\beta a)(x,\xi)|\leq C_{\alpha\beta}$ for $|\alpha|+|\beta|\leq 2n+1$. ${\mathcal A}$ is a Banach algebra for the multiplication with the norm $\|a\|_{{\mathcal A}}=\|\omega_a\|_{L^1(\R^{2n})}$.
\end{prop}

\begin{thm}[{\cite[Theorem 2.5.10]{lerner}}]\label{thm.fp.cons}
For all non-negative function $a$ defined on $\R^{2n}$ satisfying $a^{(4)}\in{\mathcal A}$, then the operator $a^w$ is semi-bounded from below. More precisely, 
$$a^w+C_n\|a^{(4)}\|_{{\mathcal A}}\geq0,$$
where $C_n$ depends only on the dimension $n$.
\end{thm}

\subsection{Proof of Theorem \ref{thm.fp}}
We shall use the partition of unity $(\varphi_Y)_{Y\in\R^{2n}}$ given in Theorem \ref{thm.partition}. Let $(\psi_Y)_{Y\in\R^{2n}}$ be a family of real-valued functions supported in $U_{Y,2r}$, equal to 1 on $U_{Y,r}$ and
\begin{equation}\label{psi}
\sup_{Y\in\R^{2n}}\|\psi_Y\|_{S(1,g)}^{(k)}\leq C(k,r,C_0).
\end{equation}
Indeed, with the same notations as in the proof of Theorem \ref{thm.partition}, the function $\psi_Y(X)=\chi_0\big(\frac{1}{2}r^{-2}g_Y(X-Y)\big)$ satisfies the requirements. Then with $a_Y=\varphi_Ya$, we write
\begin{equation}\label{eq.fp.1}
\psi_Y\sharp a_Y\sharp\psi_Y=a_Y+r_Y.
\end{equation}

\begin{lem}[Estimate for $r_Y$]\label{lem.fp.rest}
For all $k,N\in\N$, there exist $C=C(k,N,C_0)>0$, $l=l(k,N,C_0)\in\N$ such that for all $X\in\R^{2n}$, $T\in\R^{2n}$ with $g_Y(T)\leq1$,
\begin{equation}\label{eq.rest1}
|r_Y^{(k)}(X)T^k|\leq C\|a_Y\|_{S(\lambda_g(Y)^2,g_Y)}^{(l)}\big(1+g_Y^\sigma(X-U_{Y,2r})\big)^{-N}.
\end{equation}
Moreover, there exist $C_1=C_1(n,C_0,N_0)>0$, $l_1=l_1(n,C_0,N_0)\in\N$ such that 
\begin{equation}\label{eq.fp.2}
\big\|\int_{\R^{2n}}r_Y^w|g_Y|^{1/2}dY\big\|_{{\mathcal L}(L^2(\R^{n}))}\leq C_1\|a\|_{S(\lambda_g^2,g)}^{(l_1)},
\end{equation}
\end{lem}

To prove Lemma \ref{lem.fp.rest}, we use the biconfinement estimate for the remainders, the proof of which can be found in \cite[section 2.3]{lerner}.

\begin{thm}[{\cite[Theorem 2.3.4]{lerner}}, biconfinement estimate]\label{thm.biconf.r}
Let $g_1,g_2$ be two positive-definite quadratic forms on $\R^{2n}$ with $g_j\leq g_j^\sigma$. Let $a_j,j=1,2$ be $g_j$-confined in $U_j$, a $g_j$-ball of radius $\leq1$. Recall (\ref{eq.compo.asy})
$$r_p(a_1,a_2)(X)=(a_1\sharp a_2)(X)-\sum_{0\leq k<p}\frac{1}{j!}\big(i\pi[D_{X_1},D_{X_2}]\big)^j\big(a_1(X_1)a_2(X_2)\big)_{|X_1=X_2=X}.$$ 
Then for all $k,l,p\in\N$, for all $X,T\in\R^{2n}$, we have
\begin{align}
\notag&|\big(r_p(a_1,a_2)\big)^{(k)}(X)T^k|\leq A_{k,N,p}(g_1+g_2)(T)^{k/2}\Lambda_{1,2}^{-p}\\
\label{eq.biconf.r}&\qquad\qquad\times\Big(1+(g_1^\sigma\wedge g_2^\sigma)(X-U_1)+(g_1^\sigma\wedge g_2^\sigma)(X-U_2)\Big)^{-N/2}
\end{align}
with $A_{k,N,p}=C(k,N,p,n)|\!|\!|a_1|\!|\!|_{g_1,U_1}^{(l)}|\!|\!|a_2|\!|\!|_{g_2,U_2}^{(l)}$, $l=2n+1+k+p+N$ and
\begin{equation}\label{Lambda}
\Lambda_{1,2}=\inf_{T\in\R^{2n},T\neq0}\Big(\frac{g_1^\sigma(T)}{g_2(T)}\Big)^{1/2}=\inf_{T\in\R^{2n},T\neq0}\Big(\frac{g_2^\sigma(T)}{g_1(T)}\Big)^{1/2}.
\end{equation}
\end{thm}

Now we use Theorem \ref{thm.biconf.r} to prove Lemma \ref{lem.fp.rest}.

\begin{proof}[Proof of Lemma \ref{lem.fp.rest}]
By the asymptotic formula \eqref{eq.compo.asy4}, we have
$$\psi_Y\sharp a_Y=a_Y+\frac{1}{4i\pi }\underbrace{\{\psi_Y,a_Y\}}_{=0}+r_2(\psi_Y,a_Y),$$
since $\psi_Y=1$ on the support of $a_Y$.
The symbol $\psi_Y$ is $g_Y$-confined in $U_{Y,2r}$ and $a_Y$ is $g_Y$-confined in $U_{Y,r}$, and moreover, we have
$$\forall l\in\N,\quad|\!|\!|\psi_Y|\!|\!|_{g_Y,U_{Y,2r}}^{(l)}=\|\psi_Y\|_{S(1,g_Y)}^{(l)},\quad|\!|\!|a_Y|\!|\!|_{g_Y,U_{Y,r}}^{(l)}=\lambda_g(Y)^2\|a_Y\|_{S(\lambda_g(Y)^2,g_Y)}^{(l)}.$$
Applying \eqref{eq.biconf.r} to $r_2(\psi_Y,a_Y)$, we have for all $k,N\in\N$, there exist $C(k,N,n)>0$, $l(k,N,n)\in\N$ such that for all $X,T\in\R^{2n}$,
\begin{align}
\notag&|\big(r_2(\psi_Y,a_Y)\big)^{(k)}(X)T^k|\\
\notag&\quad\leq C(k,N,n)|\!|\!|\psi_Y|\!|\!|_{g_Y,U_{Y,2r}}^{(l)}|\!|\!|a_Y|\!|\!|_{g_Y,U_{Y,r}}^{(l)}g_Y(T)^{k/2}\Lambda_{1,2}^{-2}\big(1+g_Y^\sigma(X-U_{Y,2r})\big)^{-N}\\
\label{eq.rest2}&\quad\leq C(k,N,n)\|\psi_Y\|_{S(1,g_Y)}^{(l)}\|a_Y\|_{S(\lambda_g(Y)^2,g_Y)}^{(l)}g_Y(T)^{k/2}\big(1+g_Y^\sigma(X-U_{Y,2r})\big)^{-N},
\end{align}
noticing here $\Lambda_{1,2}$ defined in \eqref{Lambda} is equal to $\lambda_g(Y)$.
An analogous estimate as \eqref{eq.rest2} holds for $r_2(a_Y,\psi_Y)$. 
In our case, we write $r_Y$, which is defined in \eqref{eq.fp.1},
\begin{align*}
r_Y&=(\psi_Y\sharp a_Y-a_Y)\sharp \psi_Y+(a_Y\sharp\psi_Y-a_Y)\\
&=r_2(\psi_Y,a_Y)\sharp \psi_Y+r_2(a_Y,\psi_Y).
\end{align*}
Then the estimate \eqref{eq.rest1} follows from \eqref{eq.rest2} and \eqref{eq.biconf}.
Furthermore, for any $k, N\in\N$, there exist $C=C(k,N,n,C_0)>0$, $l=l(k,N,n,C_0)\in\N$ such that 
$$\|\bar r_Y\sharp r_Z\|_{S(1,g_Y+g_Z)}^{(k)}\leq C\| a_Y\|_{S(\lambda_g(Y)^2,g_Y)}^{(l)}\|a_Z\|_{S(\lambda_g(Z)^2,g_Z)}^{(l)}\delta_{2r}(Y,Z)^{-N}.$$
Thus we can apply Cotlar's lemma and get the estimate \eqref{eq.fp.2}.
\end{proof}

\begin{lem}[Estimate for $\psi_Y$]\label{lem.fp.rest2}
For all $k, N\in\N$, there exist $C=C(k,N,C_0)>0$, $l=l(k,N,C_0)\in\N$ such that for all $X\in\R^{2n}$, $T\in\R^{2n}$ with $g_Y(T)\leq1$,
\begin{equation}\label{lem.fp.rest2}
|(\psi_Y\sharp\psi_Y)^{(k)}(X)T^k|\leq C\big(\|\psi_Y\|_{S(1,g_Y)}^{(l)}\big)^2\big(1+g_Y^\sigma(X-U_{Y,2r})\big)^{-N}.
\end{equation}
Moreover, there exists $C_2=C_2(n,C_0,N_0)>0$ such that
\begin{equation}\label{eq.fp.3}
\|\int\psi_Y^w\psi_Y^w|g_Y|^{1/2}dY\|_{{\mathcal L}(L^2(\R^n))}\leq C_2.
\end{equation}
\end{lem}
\begin{proof}
The inequality \eqref{lem.fp.rest2} follows immediately from \eqref{eq.biconf}.
And it follows from \eqref{eq.biconf}, \eqref{psi} and \eqref{lem.fp.rest2} that for all $k,N\in\N$,
$$\|(\psi_Y\sharp\psi_Y)\sharp(\psi_Z\sharp\psi_Z)\|_{S(1,g_Y+g_Z)}^{(k)}\leq C\delta_{2r}(Y,Z)^{-N},$$
for some $C=C(k,N,n,C_0)>0$. Then by choosing $N=2N_1$ and using Cotlar's lemma, we get the estimate \eqref{eq.fp.3}.
\end{proof}

\begin{proof}[End of the proof of Theorem \ref{thm.fp}]
The symbol $a_Y$ is non-negative and uniformly in $S(\lambda_g(Y)^2,g_Y)$, so that we can apply the Fefferman-Phong inequality (Theorem \ref{thm.fp.cons}) for the constant metric $g_Y$ to get
$$a_Y^w+C(n)\|a_Y\|_{S(\lambda_g(Y)^2,g_Y)}^{(l(n))}\geq0.$$
By Proposition \ref{prop.parti.a} and Lemma \ref{lem.lambda}, we have
$$\|a_Y\|_{S(\lambda_g(Y)^2,g_Y)}^{(l(n))}\leq C(n,C_0,N_0)\|a\|_{S(\lambda_g^2,g)}^{(l(n))},$$
so that
\begin{equation}\label{eq.fp.4}
a_Y^w+C_3\|a\|_{S(\lambda_g^2,g)}^{(l(n))}\geq0.
\end{equation}
where $C_3=C_3(n,C_0,N_0)>0$, $l(n)\in\N$ are constants. Combining \eqref{eq.fp.1}, \eqref{eq.fp.2}, \eqref{eq.fp.3} and \eqref{eq.fp.4}, we obtain
\begin{align*}
a^w&=\int_{\R^{2n}}a_Y^w|g_Y|^{1/2}dY\\
&=\int_{\R^{2n}}\psi_Y^wa_Y^w\psi_Y^w|g_Y|^{1/2}dY-\int_{\R^{2n}}r_Y^w|g_Y|^{1/2}dY\\
&\geq -C_3\|a\|_{S(\lambda_g^2,g)}^{(l(n))}\int_{\R^{2n}}\psi_Y^w\psi_Y^w|g_Y|^{1/2}dY-C_1\|a\|_{S(\lambda_g^2,g)}^{(l_1)}\\
&\geq-C\|a\|_{S(\lambda_g^2,g)}^{(l)},
\end{align*}
for some $C=C(n,C_0,N_0)>0$ and $l=l(n,C_0,N_0)\in\N$. The proof of Theorem \ref{thm.fp} is complete.
\end{proof}


\bibliography{constweyl_1}
\nocite{*}
\bibliographystyle{amsplain}

\end{document}